\documentclass[10pt,oneside]{amsart}
\usepackage{color}

\newcommand{\nr}{{|\!|}}

\usepackage{amsfonts}
\newcommand{\G}{\mathrm{I}\hspace{-2.3pt}\Gamma}
 \newtheorem{thm}{Theorem}[section]
 \newtheorem{cor}[thm]{Corollary}
 \newtheorem{lem}[thm]{Lemma}
 \newtheorem{prop}[thm]{Proposition}
 \newtheorem{defn}[thm]{Definition }
 \newtheorem{rem}[thm]{Remark}
 
 \newcommand{\0}{L^2(\mathbb{R}^3)}

\begin{document}
\title[]{$L^2$ Analysis of the Multi-Configuration Time-Dependent Hartree-Fock Equations}
\author[Norbert Mauser]{Norbert J. Mauser }
\author[Saber Trabelsi]{Saber Trabelsi}
\address[Norbert Mauser]{ Wolfgang Pauli Inst. c/o Fak. f. 
Mathematik; Univ. Wien Nordbergstrasse, 15, A-1090 
Vienna, Austria } 
\address[Saber Trabelsi]{Wolfgang Pauli Inst. c/o Fak. f. 
Math.; Univ. Wien,
and CEREMADE, Univ. Paris Dauphine}
\subjclass{}
\keywords{}


\begin{abstract}
The multiconfiguration methods are widely used by quantum physicists and chemists for numerical approximation of the 
many electron Schr\"odinger equation. Recently, first mathematically rigorous results were obtained on the 
time-dependent models, e.g. short-in-time well-posedness in the Sobolev space $H^2$ for bounded interactions \cite{l2Lubich1} with initial data in $H^2$,  
in the energy space for Coulomb interactions  with initial data in the same space \cite{l2trabelsi,l2Bardos12}, as well as global well-posedness under a sufficient condition 
on the energy of the initial data \cite{l2Bardos1,l2Bardos12}. The present contribution extends the analysis by setting an $L^2$ theory 
for the MCTDHF for general interactions including the Coulomb case. This kind of results is also the theoretical foundation of ad-hoc methods used in numerical calculation when modification ("regularization") of the density matrix
destroys the conservation of energy property, but keeps invariant the mass.
\end{abstract}
\maketitle
\section{Introduction}
The Multi-Configuration Time-Dependent Hartree-Fock (MCTDHF in short) system aims to describe the time-evolution of systems  composed of fermions (e.g. electrons), that is, particles with half integer spin. More precisely, it represents a hierarchy of approximations of the linear $N$ particle (fermion) Schr\"odinger equation. 
\begin{equation}\label{exact}
i\frac{\partial}{\partial t}\Psi(t,x_1,\ldots,x_N) = H_N\:\Psi(t,x_1,\ldots,x_N),
\end{equation}
where $\Psi$  is the N-particle wavefunction which we normalize to one in $L^2(\mathbb{R}^{3N};\mathbb{C})$ since $|\Psi(t,x_1,\ldots,x_n)|^2$ is interpreted as the probability density of finding the $i^{th}$ particle in the position $x_i$ at time $t$ for all $1\leq i\leq N$. In order to account for the Pauli exclusion principle which feature the fermonic nature of the particles we are dealing with, an antisymmetry condition is imposed to the wavefunction $\Psi$ \textit{i.e.}
\[\Psi(t,x_1,\ldots,x_N)=\epsilon(\tau)\:\Psi(x_{\tau(1)},\ldots, x_{\tau(1)}),\]
for every permutation of $\{1,\ldots,N\}$ with signature $\epsilon(\tau)$. 
The $N$-body Hamiltonian of the system is then the self-adjoint operator $ {H}_N $ acting on the Hilbert 
space $L^2(\mathbb{R}^{3N};\mathbb{C})$ given by 
\begin{equation}\label{Hamiltonian}
{H}_N= \sum_{1\leq i\leq N}\left(-\frac12\Delta_{x_i} + U(x_i)\right)+
\sum_{1\leq i< j\leq N}V(|x_i-x_j|).
\end{equation}
The subscript $x_i$ of $\Delta_{x_i} $ means that the derivation is with respect to the space position of the $i^{th}$ particle
and $ U(x)$ is the "external potential", created, e.g. by nuclei localized at 
fixed positions $ R_m$
with charges $z_m>0$ for $1\leq m \leq M$. 
The last term of ${H}_N$ is the interaction potential between the electrons, which is fundamentally given
by the Couloumb interaction. Hence we have
\begin{equation}\label{potentials}
U(x) = -\sum_{m=1}^M\frac{z_m}{|x-R_m|}\quad\text{and}\quad V(|x-y|) =\frac{1}{|x-y|}.
\end{equation}
We use the so-called atomic units for writing the Hamiltonian $H_N$ and the N-particle Schr\"odinger equation; hence the Planck constant $\hbar$, the mass of the electrons, the elementary charge and the factor $\frac{1}{4\pi\epsilon_0}$ are all set to one with $\epsilon_0$ denotes the dielectric constant of the vacuum. For the sake of simple notation we omit the spin variable, taking the spin into account does not add any mathematical difficulties.

A particular advantage over the simple Time-Dependent Hartree-Fock (TDHF) method (see e.g. \cite{Weeny,l2Chadam,l2Bardos2,l2Bardos12} 
 is that MCTDHF can handle "correlation", an essential concept of many electron systems that vanishes (by definition) 
for TDHF. We refer the reader to \cite{l2Gottlieb,l2Gottlieb2} for more details.
\vskip6pt
MCTDHF methods are widely used for numerical calculation of the dynamics of few electron systems in quantum physics/chemistry (see e.g. \cite{l2Scrinzi1,l2Scrinzi2,l2Scrinzi3,l2Scrinzi} and references therein, also \cite{l2Meyer,l2Meyer2,l2Meyer1,l2Meyer3} for the MCTDH case). From a theoretical point of view, excellent results were obtained for the stationary case \cite{l2Lebris,l2Fries,l2Lewin}. However, the time-dependent case still poses  serious open problems for the mathematical analysis. In fact, even the global-in-time existence of solutions without recourse to extra assumption as in  \cite{l2Bardos1,l2Bardos12} is not proved yet. 
\vskip6pt
The MCTDHF system for the approximate wavefunction $\Psi_{\mathrm{MC}}(t,x_1,...,x_N) $ 
is composed of $K \ge N$ non-linear Schr\"odinger-type evolution equations (for ``the orbitals", 
as a dynamic basis for an expansion by ``Slater determinants") coupled with $r:= {K\choose N}$ first order differential 
equations (for `` the coefficients" $C$). 
The many particle  wavefunction $\Psi$ can be well approximated by such linear combinations of 
Slater determinants, much better than by the simple TDHF method that corresponds to the limiting case $K = N$.
In principle, for fixed $N$, the MCTDHF equations yield a hierarchy of models with increasing accuracy with 
increasing $K$, in the sense that many particle wavefunction constructed from the solution 
of MCTDHF converges (in some sense) toward the exact solution $\Psi$ with increasing $K$. 
However, especially in the time-dependent case, there is no proof for this seemingly ``obvious" property.
\vskip6pt
Let us now formulate the MCTDHF equations which are much more complicated to write down than the 
"usual NLS" like cubic NLS or "Schr\"odinger-Poisson".  
For a short and readable introduction to the multiconfiguration time-dependent Hartree-Fock system, 
we refer the reader to \cite{l2Lubich1,l2trabelsi,l2Bardos1,l2Bardos12},
or \cite{l2Lebris,l2Fries,l2Lewin} focussed on the stationary case. 
First of all, we introduce the set $\mathcal{F}_{N,K}$ of "coefficents and orbitals" $(C,\Phi)$ 
\begin{eqnarray} \label{FNKdef}
\mathcal{F}_{N,K}=\Biggl\{ C=(C_{\sigma})_{\sigma\in \Sigma_{N,K}} \in \mathbb{C}^{r},\quad\sum_{ \sigma\in\Sigma_{N,K}}|C_{\sigma}|^2=1\Biggr\} \times \\  
\Biggl\{ \Phi=(\phi_1,\ldots,\phi_K )\in L^2(\mathbb{R}^3)^K, \quad \int_{\mathbb{R}^3}\phi_i\, \bar\phi_j\,dx=\delta_{i,j}\Biggr\},
\end{eqnarray}
with $ \delta_{i,j}$ being the Kronecker delta, the bar denotes the complex conjugation and 
\begin{equation*}
\Sigma_{N,K}=\Big\{\sigma=\{\sigma(1)<\ldots < \sigma(N)\}\subset\{1,\ldots,K\}\Big\},\qquad  |\Sigma_{N,K}|={K\choose N}:=r.
\end{equation*}
That is, the range of the family of increasing mappings 
$ \sigma:\quad\{1,\ldots,N\}\longrightarrow\{1,\ldots,K\} $ for $ 1\leq N \leq K $ with integers $N$ and $K$ and we shall not make difference between mappings and sets $\sigma$ for simplicity.  Now, given $ \sigma \in \Sigma_{N,K} $, we define the associated {\it Slater determinant} as follows
\begin{equation*}\label{Slater}
\Phi_{\sigma}(x_1,\ldots,x_N)=\frac{1}{\sqrt{N!}} \left|\begin{array}{ccc}
\phi_{\sigma(1)}(x_1)  &  \ldots &  \phi_{\sigma(1)}(x_N) \\
  \vdots &   &  \vdots \\
 \phi_{\sigma(N)}(x_1) &\ldots   &   \phi_{\sigma(N)}(x_N)
\end{array}\right|.
\end{equation*}
That is, the determinant built from the $ \phi_i$'s such that $ i\in \sigma$. 
The $ \phi_i$'s being an orthonormal family, the factor $\frac{1}{\sqrt{N!}} $ ensures then  the normalization $ \|\Phi_\sigma\|_{L^2(\mathbb{R}^{3N})}=1$. 
The MCTDHF wavefunction reads now
\begin{equation}\label{PsiMC}
\Psi_{\mathrm{MC}}(t,x_1,\ldots,x_N) := \Psi_{\mathrm{MC}}(C(t),\Phi(t)) 
= \sum_{\sigma\in \Sigma_{N,K}} \:C_\sigma(t)\:\Phi_\sigma(t,x_1,\ldots,x_N).
\end{equation}
In the sequel, we shall use the same notation for operators acting on $L^2(\mathbb{R}^3)$ 
and diagonal matrices of operators acting on vectors in $L^2(\mathbb{R}^3)^K$. 
The equations of motion associated to the {\it ansatz} \eqref{PsiMC} correspond to a variational procedure 
combined with space rotations thanks to the gauge invariance of the system we are dealing with. We refer to \cite{l2Meyer,l2Lubich1} and particulary to \cite{l2Bardos12}. 
In our formulation we obtain the following coupled system :\\ 
\begin{equation} \label{Sdef}
\mathcal{S}:\quad \left\lbrace 
\begin{array}{ll}
  & i\frac{d}{dt}\:C (t) = \mathbb{K}[\Phi](t)\:C(t),        \\
  &      \\
  &i\:\G[C(t)]\:\frac{\partial}{\partial t}\:\Phi(t,x) = \G[C(t)]\:\left[-\frac{1}{2}\Delta_x +U(x) \right]\:\Phi(t,x) + (I-\mathbf{P}_\Phi)\:\mathbb{W}[C,\Phi](t,x)\:\Phi(t,x).      
\end{array}
\right.
\end{equation}
\indent \\ Let us now define the different operators and matrices involved in the system $\mathcal{S}$. 
First, $\mathbb{K}[\Phi]$ denotes an $r\times r$ Hermitian matrix depending only on the orbitals $\phi_i$
and the binary interaction $V$ 
and is indexed by the sets $\sigma,\tau \in \Sigma_{N,K}$ as follows
\begin{eqnarray} \label{Kdef}
\mathbb{K}[\Phi]_{\sigma,\tau}(t)= \frac12\:
\sum_{i,j\in\tau,\;k,l \in \sigma}&&\:(1-\delta_{i,j})(1-\delta_{k,l})\:\delta_{\tau\setminus\{i,j\},\sigma\setminus\{k,l\}}(-1)^{\tau}_{i,j}\;(-1)^{\sigma}_{k,l}\times\\&\times&\int_{\mathbb{R}^3\times \mathbb{R}^3 }
\phi_i(t,x)\:\overline{\phi}_k(t,x) \:V(|x-y|)\: \phi_j(t,y)\:\overline{\phi}_l(t,y)\:\:dx\: dy,
\end{eqnarray}
where the symbols $(-1)^{\sigma}_{k,l} = \pm1$ are not relevant for analysis and are given by
\begin{equation}\label{def-sigmaij}
 (-1)^{\sigma}_{k,l}=
\left\{
\begin{array}{ll}
(-1)^{\sigma^{-1}(k) +\sigma^{-1}(l) +1} & \mathrm{if}  \,\, k<l, \\
 \\
(-1)^{\sigma^{-1}(k) +\sigma^{-1}(l)} & \mathrm{if}\,\,  k>l,
  \end{array}
\right.
 \end{equation}
with $\sigma^{-1}(k)$ being the position of the entry $k$ in the set $\sigma$.
\vskip6pt
The matrix $\G[C(t)] $ is a $K\times K$ Hermitian "density matrix" depending only on the coefficients as follows \\
\begin{equation}\label{densitymatrix}
\G[C(t)]_{i,j} = \sum_{\stackrel{\sigma,\tau \in \Sigma_{N,K}}{\sigma\setminus\{i\}= \tau\setminus\{j\}}} (-1)^{\sigma^{-1}(i) + \tau^{-1}(j)} \:\overline{C}_\sigma(t)\:{C}_{\tau}(t).
\end{equation}
The projection operator $\mathbf{P}_\Phi$ is given by
\begin{equation}\label{Pdef}
\mathbf{P}_\Phi(\psi)=
\sum_{i=1}^K \int_{\mathbb{R}^3}\:\psi(t,x)\:\overline{\phi}_{i}(t,x)\:dx\: \phi_i(t,x),
\end{equation}
for all $\psi \in \0$. That is the orthogonal projector onto the space spanned by the $\phi_i'$s. 
Finally, the $K\times K$ Hermitian matrix $\mathbb{W}[C,\Phi](t,x)$ is given by the entries  
\begin{equation} \label{Wdef}
{\mathbb{W}[C,\Phi]} _{i j}(x) = 2\:\sum_{k,l=1}^K \gamma_{jkil} \,\int_{\mathbb{R}^3}\:\phi_k(y)\:V(|x-y|)\:\overline{\phi}_l(y)\:dy, 
\end{equation}
with 
\begin{equation} \label{gammaijkldef}
\gamma_{ijkl}=\frac 12\:(1-\delta_{i,j})(1-\delta_{k,l})\: \sum_{\substack 
{\sigma,\tau\:\vert \:i,j\in \sigma,\,k,l\in\tau
\\ 
\sigma\setminus\{i,j\}= \tau\setminus\{k,l\} 
}}
(-1)^{\sigma}_{i,j} (-1)^{\tau}_{k,l} \:{C}_{\sigma}\: \overline{C}_\tau.
\end{equation}
Observe that the potential $V$ appears in the definition of $\mathbb{W}[C,\Phi]$ and $\mathbb{K}[\Phi]$, 
but that neither the kinetic energy operator nor the interaction $U$ appear in $\mathbb{K}[\Phi]$ 
and hence in the equation for the $C(t)$.
We refer the reader to \cite{l2Bardos12} for more details on this formulation that corresponds to a carefully chosen gauge. Hence we have to deal with a strongly non-linear coupled system of $r = {K\choose N}$ first order ODEs and $K$ Schr\"odinger-type PDEs. 
\vskip6pt
A preliminary result on the existence and uniqueness of solutions to the MCTDH system $\mathcal{S}$ 
(i.e. the adequate version of the multiconfiguration models for bosons) 
has been established in \cite{l2Lubich1} under the drastic simplification of bounded and smooth interaction potential $V$
(and $U\equiv 0$), where $H^2$ regularity for a solution associated to initial data in the Sobolev space $H^2$
was shown to hold up to a certain positive time.  
Indeed, for such potentials it should be straightforward to get $L^2$ solutions, too.
\vskip6pt
For the case of the Coulomb potential, some results have recently been obtained
in \cite{l2trabelsi,l2Bardos12} with less regularity of initial data, namely in $H^1$. 
 However, all these results are local-in-time in the sense that the existence, uniqueness and regularity persist 
only as long as the density operator associated to the system remains of full rank $K$. 
That is, the matrix $\G[C(t)]$ remains invertible.  In case of a ``loss of rank" at a certain time $T^*$, 
the well-posedness holds only locally-in-time, until $T^*$. 
However, in \cite{l2Bardos12,l2Bardos1} it is shown how the global-in-time existence can be assured under an 
assumption on the energy of the initial state $\Psi_{\mathrm{MC}}(C(t=0),\Phi(t=0))$.
\vskip6pt
The purpose of this contribution is, essentially, to establish an $L^2$ existence and uniqueness result 
of solutions for the Cauchy problem associated to the MCTDHF equations, the system $\mathcal{S}$. 
For the binary interaction $V$ and the external potential $U$ we allow for a large class including 
the Coulomb potential :
\begin{eqnarray} \label{UVcond1}
&& U(|x|),V(|x|)\in L^{p_1}(\mathbb{R}^3,\mathbb{R}) + L^{p_2}(\mathbb{R}^3,\mathbb{R}),\qquad p_1,p_2>\frac32.
\end{eqnarray}
In order to make the proofs less technical, we shall take from now on 
$ U,V\in L^{p}(\mathbb{R}^3,\mathbb{R})$  with $p>\frac32$ omitting the decomposition in two parts. 
One can easily see that the same proof holds {\it mutatis mutandis} in the decomposed case, for
the price of technically heavier and longer formulation of the estimates. 
One more simplification will be used in the sequel, we will usually write $U,V \in L^{p}(\mathbb{R}^3)$ 
with $p>\frac32$, however, this has to be understood in the sense that 
$U \in L^{p}(\mathbb{R}^3,\mathbb{R})$ and $V \in L^{q}(\mathbb{R}^3,\mathbb{R})$ 
with different $p$ and $q$ satisfying $p,q>3/2$. 
\vskip6pt
Observe that the assumptions above hold true in the Coulomb case \eqref{potentials}, 
as can be seen using the following cut-off 
\begin{eqnarray} \label{chidef}
\chi(r) \left\lbrace \begin{array}{l l l}
=1 & for &  0 \leq r \leq 1 \\
\in [0,1] & for &  1 \leq r \leq 2 \\
=0 & for & 2 \leq r \\
\end{array} \right. \quad\text{and set}\quad \left\{
\begin{array}{ll}
  &V_1(|x|)=\frac{\chi(|x|)}{|x|}\in L^p(\mathbb{R}^3)      \\
  &      \\
  &  V_2(|x|)=\frac{1-\chi(|x|)}{|x|}L^\infty(\mathbb{R}^3)    
\end{array}\right.
\end{eqnarray}
with 
\[\frac{3}{2} \leq p < 3.\]
\vskip6pt
Now, let us introduce
\begin{equation}\label{regularfiber}
\partial\mathcal{F}_{N,K} =\bigg\{ (C,\Phi) \in \mathcal{F}_{N,K}\quad :\quad \mathrm{rank}(\G[C]) =K\bigg\}.
\end{equation}
With these definitions, we are able to recall a result from \cite{l2trabelsi,l2Bardos12} that 
will be useful in the sequel. 
\begin{thm}\label{H1thm}
Let $(C^0,\Phi^0) \in \partial\mathcal{F}_{N,K}\cap H^1(\mathbb{R}^3)^K$ be an $H^1$ initial data, 
$U,V \in L^{d}(\mathbb{R}^3)$ with $d>\frac32$. Then, there exists a time $T^\star >0$ such that the MCTDHF system $\mathcal{S}$ admits a unique maximal solution $(C(t),\Phi(t)) $ in
\begin{equation*}
C^1\big([0,T^\star);\mathbb{C}\big)^r \times
\Big(C^1\big([0,T^\star);H^{-1}(\mathbb{R}^3)\big)^K \cap 
C^0\big([0,T^\star);H^{1}(\mathbb{R}^3)\big)^K\Big)
\end{equation*}  
depending continuously on the initial data $(C^0,\Phi^0)$. Moreover, for all $0\leq t< T^\star$, we have 
\begin{itemize}
\item $(C(t),\Phi(t))\in \partial\mathcal{F}_{N,K}$,\\
\item $\left\langle \Psi_{\mathrm{MC}}(C(t),\Phi(t))|H_N|\Psi_{\mathrm{MC}}(C(t),\Phi(t))\right\rangle_{L^2(\mathbb{R}^{3N})} = \left\langle \Psi_{\mathrm{MC}}(C^0,\Phi^0)|H_N|\Psi_{\mathrm{MC}}(C^0,\Phi^0)\right\rangle_{L^2(\mathbb{R}^{3N})}. $
\end{itemize}
Finally, either $T^\star=+\infty$ or $T^\star< +\infty$ and $\int_{0}^{T^\star} \nr \G[C(t)]^{-1}\nr^{3/2} dt=+\infty$. 
In particular
\[\lim\sup_{t\rightarrow T^\star}\:\nr \G[C(t)]^{-1}\nr^{\frac32}=+\infty.\]
\end{thm}
This Theorem then yields a local well-posedness result to the MCTDHF system. 
However, observe that the result is possibly global since, for the time being, there is no indication
that $T^\star$ is necessarily finite \cite{l2trabelsi,l2Bardos1,l2Bardos12}. As noticed before, the special case $N=K$ corresponds to a single Slater determinant {\it ansatz}. In particular, the set $\mathcal{F}_{N,N}$ coincides with $\partial\mathcal{F}_{N,N}$ and becomes a smooth manifold actually. In other words, the matrix $\G[C] $ reduces to a globally invertible matrix since it becomes the $N \times N$ identity matrix. 
Therefore, Theorem \ref{H1thm} ensures the global existence of a unique solution to 
the TDHF system, that is, $T^\star=+\infty$ and it improves then previous results obtained in \cite{l2Chadam}. Actually, the proof of the energy conservation presented in \cite{l2Bardos12} is much more readable than the one there. 

Finally, recall that this result (and the one we shall present in the next section) is valid in the case of the MCTDH. The so-called Schr\"odinger-Poisson system (SPS), which coincides with the Hartree system in the special
case of "Bose Einstein condensation" when Coulombic interaction is used, 
can be also obtained as a limiting case of the MCTDH and our result applies obviously to this model, too. 
\section{Main result}
In \cite{l2Zagatti}, an $L^2$-Theory to a TDHF type model is established. 
However, that model is a peculiar mixed state formulation.
Usually, Hartree-Fock is characterized by a finite number of $N$ equations with occupation numbers equal to $1$
which is, of course, a completely different paradigm than a mixed state for a one particle model. That is, the eigenvalues of $\G[C(t)] $ are one for all time $t\geq 0$. 
However, the result obtained in  \cite{l2Zagatti} obviously also holds for finite $N$ (even if the author does not explicitly remark that) and can obviously be adapted to the usual TDHF setting.
In our work this TDHF result is improved. 
\vskip6pt
Independently,  Castella established an $L^2$ theory of the mixed state Schr\"odinger-Poisson system \cite{l2Castella}. 
More precisely, he studied  a system of infinitely coupled Schr\"odinger equations with self-consistent Coulomb potential.  The initial data needs only an $L^2$ bound, so the initial kinetic energy can possibly 
be infinite. 
Moreover, he obtained a blow-up (resp. decay) estimates for the solution as time goes to zero (resp. infinity). 

Our work is inspired by Castella's results and Strichartz techniques. Our result on MCTDHF applies also to the TDHF and the TDH "pure state" case. 
We prove the following 
\begin{thm} \label{L2thm} 
Let the potentials $U,{V}\in L^{d}(\mathbb{R}^3) $ with $d> \frac32$ in the sense of (\ref{UVcond1})
and let $(C^0,\Phi^0) \in \partial\mathcal{F}_{N,K}$ 
be an $L^2$ initial data with full rank (cf (\ref{FNKdef}), (\ref{regularfiber})).\\  
Then there exists a time $T^\star>0$ (possibly $T^\star=+\infty$) 
such that the MCTDHF system $\mathcal{S}$ (\ref{Sdef}) admits solutions $(C(t),\Phi(t))$ satisfying 
\begin{itemize}
\item $C \in C^1([0,T^\star), \mathbb{C})^r\quad \text{and}\quad \Phi \in C^0([0,T^\star), L^2(\mathbb{R}^3))^K. $ 
\end{itemize}
Moreover, for all $2\leq q< 6$
\begin{itemize}
\item[$\mathrm{i})$] $\Phi(t) \in L^{\frac{4q}{3(q-2)}}([0,T], L^q(\mathbb{R}^3))^K $. \\
\item[$\mathrm{ii})$] The solution $(C(t),\Phi(t))$ is unique in the class
\[ L^\infty([0,T], \mathbb{C})^r \times L^{\infty}([0,T], L^2(\mathbb{R}^3))^K\cap L^{\frac{4q}{3(q-2)}}([0,T], L^q(\mathbb{R}^3))^K, \]
for all $ T< T^\star$.\\
\item[$\mathrm{iii})$] $(C(t),\Phi(t))\in \partial\mathcal{F}_{N,K}$ for all $t\in [0,T^\star)$. 
\end{itemize}
\end{thm}
Hence, this result establishes an $L^2$ theory for the solution $\Psi_{\mathrm{MC}}(C(t),\Phi(t))$ 
(\ref{PsiMC}) of the MCTDHF equations 
as long as the density matrix $\G[C(t)]$ remains of full rank. 
Note that our theorem yields a global $L^2$ theory for the TDHF and the TDH models 
(the "pure state"  versions of the models studied in \cite{l2Zagatti,l2Castella} as mentioned above). 
\vskip6pt
The paper is structured as follows. In the next section, we collect some well-known tools, 
like Strichartz estimates and the properties associated to the semigroup $\mathbf{U}(t)$ generated 
by $i\frac{1}{2} \Delta$ on $L^2(\mathbb{R}^3)$. 
Moreover, we prove a local existence result using a standard contraction argument in an adequate space $\mathcal{X}_{T'}^{p,q}$ for a given reals $p,q$ and a nonnegative time $T'$. In section 3, we prove that this local solution satisfies an {\it a priori} estimate which will be crucial in order to prolongate the solution beyond $ T'$. Section 4 is dedicated to the proof of the main result \ref{L2thm}. 
Finally, we make some comments and application of this Theorem.
\section{A few technical Lemmata}
First of all, let us specify the notation we will use throughout this paper and recall some well known tools. The real $p'$ will be the conjugate of $p$ with $1\leq p\leq \infty $, that is, $\frac{1}{p}+\frac{1}{p'}=1$. By abuse of notation, we denote $L^p=L^p(\mathbb{R}^3,\mathbb{C}) $ but also $L^p=L^p(\mathbb{R}^3,\mathbb{R})$ when there is no confusion. The same notation will be used for $L^p(\mathbb{R}^3,\mathbb{C}) ^K$ and will be specified explicitly when necessary. The associated norms will be denoted $\nr\cdot\nr_{L^p}$. The same conventions are adopted for the Sobolev spaces $H^1$. $(\mathbf{U}(t))_{t\in \mathbb{R}}$ is the group of isometries $(e^{\frac{i}{2}t\Delta})_{t\in \mathbb{R}}$ generated by $\frac{i}{2}\Delta$ on $L^2(\mathbb{R}^3,\mathbb{C})$. 
Finally, $\kappa$ will be an auxiliary positive constant depending on $N$ and $K$. 
Also "$\mathrm{const.}$" will denote generic constants depending on quantities that will be indicated explicitly when necessary. Next, for a given $V(|x|)$, let for all $ 1\leq i,j\leq K$
\[
\mathbb{V}[\Phi]_{i,j}(t,x):= \int_{\mathbb{R}^3}\:\:\phi_i(t,y)\:V(|x-y|)\:\overline{\phi}_j(t,y)\:dy
\]
and
\[
D_V[f,\overline{g}](t) = \int_{\mathbb{R}^3\times \mathbb{R}^3} f(t,x) \:V(|x-y|)\:\overline{g}(t,y)\:dx\:dy.
\]
From this point onward, we shall omit the dependence on $t$ and $x$ when the context  is clear. Finally, for a given $T>0$ we denote 
$
L^{p,q}_T=L^p([0,T],L^q).
$
Now, 
\begin{defn}\label{admissible}
 The pair of reals $(p,q)$ is said to be {\it admissible}, we denote $(p,q)\in \mathcal{A} $, if and only if the following relation holds true.
\[
\frac{2}{3\:p}=(\frac{1}{2}-\frac{1}{q}) \quad \text{and} \quad 2\leq q < 6.
\]
\end{defn}
Then, we are able to recall 
\begin{lem}\label{Strilem} Let $ 0 < t \leq T $, then for all $ (a,b),(p,q) \in \mathcal{A},\phi \in L^2 $ and $\varphi \in L^{a',b'}_T$, there exists $\rho(a)$ and $ \rho(a,p) $ such that 
\begin{equation}\label{S2}
  \nr \mathbf{U}(t)\:\phi\nr_{L^{p,q}_T} \leq \rho(p) \:\nr\phi\nr_{L^2}\quad \text{and}\quad  \nr \int_{0}^{t} \mathbf{U}(t-s)\:\varphi(s)ds  \nr_{L^{p,q}_{T}} \leq \rho(p,a)\:\nr \varphi\nr_{L^{a',b'}_T}.
 \end{equation}
\end{lem}
\begin{proof} 
This is a classical Lemma and we refer the reader to \cite{l2Ginibre,l2Ginibre1,l2Strich,l2Caz}. 
\end{proof}
The first inequality appearing in \eqref{S2} describes a notable smoothing effect. In particular it tell us that for all $t\in \mathbb{R}$ and $\phi\in L^2$, we have obviously $\mathbf{U}(t)\:\phi \in L^p$. The second inequality is crucial when dealing with non-linearities in the framework of Schr\"odinger type equations. Indeed, without loss of generality we write the following generic Duhamel formula 
\[\psi(t)= \mathbf{U}(t)\:\phi -i\int_0^t \mathbf{U}(t-s) \:f(\psi(s)),\]
for a given functional $f$. Then, the first inequality of \eqref{S2} allows to control 
the $L^2$ norm of $\mathbf{U}(t)\:\phi $ in terms of the $L^2$ norm of $\phi$. 
However, it is merely impossible to control the $L^2$ norm of $\int_0^t \mathbf{U}(t-s) \:f(\psi(s))\:ds$ in terms of the one of $\psi$ for general non-linearities $f$. The second inequality of \eqref{S2} will, then, enable us to control the $L^{p,q}_T$ norm of $\int_0^t \mathbf{U}(t-s) \:f(\psi(s))\:ds$ for a given $T>0$ and a couple of reals $(p,q)\in \mathcal{A}$ in terms of the $L_T^{p',q'}$ norm of $f$ and allows to conclude. 
\vskip6pt
The Duhamel formula associated to the MCTDHF system $\mathcal{S}$ for a given initial data $(C^0,\Phi^0)$ is written as follows for all time $t$ such that $\G[C(t)] $ is invertible,
\begin{eqnarray}
\lefteqn{\begin{bmatrix}C(t) \\ \\ \Phi(t) \end{bmatrix}=\begin{bmatrix}C^0 \\ \\ \mathbf{U}(t)\:\Phi^0\end{bmatrix}\nonumber }\\ &&-i\int_0^t\begin{bmatrix}\mathbb{K}[\Phi(s)]\,C(s)\\ \\  \mathbf{U}(t-s)\:\Big[U\:\Phi(s) +\G[C(s)]^{-1}\:(I-\mathbf{P}_\Phi)\:\mathbb{W}[C(s),\Phi(s)]\:\Phi(s)\Big]
\end{bmatrix} \,ds\label{duhamel}.
\end{eqnarray} 
\begin{rem}
The potential $U$ being time-independent, we chose for simplicity to add it to the non-linear part. An alternative way to proceed consists in considering the linear PDE  $ i\frac{\partial}{\partial t}\:u(t,x) = -\frac12 \Delta\:u(t,x) + U(x)\:u(t,x)$ and find adequate real $p$ for $U\in L^p$ such that one can associate to this flow  a propagator that satisfies Strichartz-type estimates \eqref{S2}. We refer the reader to, e.g., \cite{l2Yajima}.  
\end{rem}
Next, formal functional analysis calculation leads to  
\begin{lem}\label{genericestimates}
Let $U, {V} \in L^d$, $(p,q),(p_i,q_i) \in \mathcal{A}$ for $i=1,\ldots,4$ and $T>0$. Then, 
\begin{eqnarray}
&& \nr U\:\phi_1\nr _{L^{p',q'}_T} \leq T^{\frac{3}{2}\left(\frac1q+\frac{1}{q_1}\right)-\frac12}\: \nr U\nr_{L^d}\:\nr\phi_1\nr_{L^{p_1,q_1}_T}, \quad \frac1q +\frac{1}{q_1}=1- \frac{1}{d},\label{estim1}\\
&& \nonumber \\
 && \nr\mathbb{V}[\Phi]_{1,2}\:\phi_3 \nr_{L^{p',q'}_T} \leq T^{\frac{3}{2}\left(\frac1q+\sum_{k=1}^3\frac{1}{q_k}\right)-2}\: \nr V\nr_{L^d}\:\prod_{i=1}^3\:\nr\phi_i\nr_{L^{p_i,q_i}_T},\quad \frac1q+\sum_{i=1}^3\frac{1}{q_i}=2-\frac1d, \label{estim2}\\
 && \nonumber \\
 && \left|\int_0^T\:D_V(\phi_1(t)\:\overline{\phi}_2(t),\phi_3(t)\:\overline{\phi}_4(t))\:dt\right| \leq T^{\frac{3}{2}\sum_{k=1}^4\frac{1}{q_k}-2}\: \nr V\nr_{L^d}\:\prod_{i=1}^4\:\nr\phi_i\nr_{L^{p_i,q_i}_T},\nonumber\\ &&\qquad \qquad \qquad \qquad \qquad \qquad\qquad \qquad \qquad\sum_{i=1}^4\frac{1}{q_i}=2-\frac1d.\label{estim3}
\end{eqnarray}
\end{lem}
\begin{proof}
The proof is nothing but a straightforward calculation based on the well-known H\"older and Young inequalities 
in space and time and we refer to \cite{l2Zagatti} for similar estimates. 
\end{proof}
An immediate corollary of Lemma \ref{genericestimates} is the following 
\begin{cor} \label{corestimates}
Let $U,{V} \in L^d, d> \frac32$ and $(p,q=\frac{2d}{d-1})\in \mathcal{A}$. Then, for all $T>0$, we have
\begin{eqnarray}
&& \nr U\:\Phi\nr _{L^{p',q'}_T} \leq \kappa\:T^{\frac{3}{q}-\frac12}\: \nr U\nr_{L^d}\:\nr\Phi\nr_{L^{p,q}_T}, \label{potestimate}\\
&&\nonumber \\
&&\left| \int_{0}^T \:\mathbb{K}[\Phi(t)]\:dt \right| \leq\kappa\: T^{\frac{3}{q} -\frac12} \:\nr V\nr_{L^d}\:\nr\Phi\nr^2_{L^{\infty,2}_T}\:\nr\Phi\nr^2_{L^{p,q}_T}, \label{first}\\
&&\nonumber  \\
&&\nr \mathbb{W}[C,\Phi]\:\Phi\nr _{L^{p',q'}_T} \leq \kappa\: T^{\frac{3}{q} -\frac12} \:\nr V\nr_{L^d}\:\nr C\nr^2_{\mathbb{C}^r}\:\nr\Phi\nr^2_{L^{\infty,2}_T}\:\nr\Phi\nr_{L^{p,q}_T},\label{second}\\
&&\nonumber \\
&& \nr \mathbf{P}_\Phi\:\mathbb{W}[C,\Phi]\:\Phi\nr _{L^{1,2}_T} \leq \kappa\: T^{\frac{3}{q} -\frac12} \:\nr V\nr_{L^d}\:\:\nr\Phi\nr^3_{L^{\infty,2}_T}\:\nr\Phi\nr^2_{L^{p,q}_T}.\label{third}
\end{eqnarray}
\end{cor}
\begin{proof}
The proof relies on the estimates  \eqref{estim2} and \eqref{estim3} of Lemma \ref{genericestimates}. The assertion \eqref{potestimate} is easy, in fact one sets $(p_1,q_1)= (p,q) \in \mathcal{A}$ and gets the result. Next, given a potential $V\in L^d$, the matrix $\mathbb{K}$ involves elements of type $D_V(\phi_i\overline{\phi}_j,\phi_k\overline{\phi}_l) $ for which one use the estimate \eqref{estim3} by setting, for instance, $(p_3,q_3)= (p_4,q_4)=(\infty,2) \in \mathcal{A}$ and $(p_1,q_1) =(p_2,q_2) =(p,q) \in \mathcal{A}$. Also, the vector $\mathbb{W}[C,\Phi]\:\Phi$ involves terms of type $\mathbb{V}[\Phi]_{i,j}\phi_k $ that can be handled using the estimate \eqref{estim2}. In fact, we set for instance $(p_2,q_2) =(p_3,q_3) = (\infty,2) \in \mathcal{A}$ and $(p_3,q_3)=(p,q) \in \mathcal{A}$. Finally, the vector $\mathbf{P}_\Phi\:\mathbb{W}[C,\Phi]\:\Phi$ involves terms of type $D_V(\phi_i\overline{\phi}_j,\phi_k\overline{\phi}_l)\:\phi_l$. Observe that $D_V(\phi_i\overline{\phi}_j,\phi_p\overline{\phi}_l)$ is a time-dependent scalar, thus, estimating the left hand side in an $L^{p',q'}_T$ leads automatically to an $L^{q'}_x$ norm on $\phi_l$ in the right hand side. For convenience, we estimate this term in $L^{1,2}_T$ in order to get an $\nr\phi_l\nr_{L^{\infty,2}_T}$ and use the same choice as in \eqref{first}.
\end{proof}
\begin{rem}
From the one side, observe that for $d>\frac{3}{2}$, we have obviously  $ 2\leq q=\frac{2d}{d-1} < 6 $. Moreover, the estimates of Corollary \ref{corestimates} involve $T^\alpha $ with power $\alpha >0 $ so that $T^\alpha \rightarrow 0 $ as $T\rightarrow 0$. Indeed, since $2\leq q < 6$, we have $0< \frac3q-\frac12 \leq 1$.  This observation will be crucial in the sequel. From the opposite side, assume $V$ bounded, that is, $d=\infty$, then, the estimates of the Corollary \ref{corestimates} are valid with $q=2$ and $\alpha$, the power of $T$, equal to $1$.
\end{rem}
Next, given $T>0$ and $(p,q) \in \mathcal{A}$, we define the spaces
\[ \mathcal{Z}^{p,q}_T = L^{\infty,2}_T \cap L^{p,q}_T,\quad \mathcal{X}^{p,q}_T= \mathbb{C}^r \times \mathcal{Z}^{p,q}_T,\]
endowed with the norms 
\[\nr\phi\nr_{\mathcal{Z}^{p,q}_T} = \nr\phi\nr_{L^{2,\infty}_T} +\nr\phi\nr_{L^{p,q}_T},\quad \nr C,\Phi\nr_{\mathcal{X}^{p,q}_T}=  \nr C \nr_{\mathbb{C}^r} + \nr\Phi\nr_{\mathcal{Z}^{p,q}_T}.\]
A topology on $\mathcal{Z}^{p,q}_T$ and $\mathcal{X}^{p,q}_T$ being defined, we are able to prove the following 
\begin{lem}\label{Lipbounds}Let $U,{V} \in L^d, d> \frac32, (p,q)\in \mathcal{A}$ such that $q=\frac{2d}{d-1}$. Then, for all $T>0$, we have
\begin{eqnarray}
&&\nr\mathbb{K}[\Phi(t)]\:C(t) -\mathbb{K}[\Phi'(t)]\:C'(t)  \nr_{L^1([0,T])} \leq \mathrm{const}_1\:T^{\frac3q-\frac12}\: \nr(C,\Phi)-(C',\Phi')\nr_{\mathcal{X}^{p,q}_T}, \label{Lip0}\\
&& \nonumber\\
&& \nr U\:(\Phi-\Phi')\nr _{L^{p',q'}_T} \leq \kappa\:T^{\frac{3}{q}-\frac12}\: \nr U\nr_{L^d}\:\nr\Phi-\Phi'\nr_{\mathcal{Z}^{p,q}_T}, \label{Lip01}\\
&\nonumber \\
&& \nr \mathbb{W}[C,\Phi]\:\Phi - \mathbb{W}[C',\Phi']\:\Phi' \nr _{L^{p',q'}_T} \leq \mathrm{const}_2\:T^{\frac3q-\frac12}\: \nr(C,\Phi)-(C',\Phi')\nr_{\mathcal{X}^{p,q}_T},\label{Lip1}\\
&& \nonumber \\
 &&\nr \mathbf{P}_\Phi\:\mathbb{W}[C,\Phi]\:\Phi-\mathbf{P}_{\Phi'}\:\mathbb{W}[C',\Phi']\:\Phi'\nr _{L^{1,2}_T}\leq \mathrm{const}_3 \:T^{\frac3q-\frac12}\: \nr(C,\Phi)-(C',\Phi')\nr_{\mathcal{X}^{p,q}_T},    \label{Lip2}
\end{eqnarray}
with $\mathrm{const}_1,\mathrm{const}_2$ and $\mathrm{const}_3$ depending on \
\begin{equation*}
N,K, \nr U\nr_{L^d},\nr{V}\nr_{L^d},\nr\Phi\nr_{L^{\infty,2}_T},\nr\Phi'\nr_{L^{\infty,2}_T},\nr C\nr_{\mathbb{C}^r},\nr C'\nr_{\mathbb{C}^r},\nr\Phi\nr_{L^{p,q}_T} \:\:\text{and}\:\: \nr\Phi'\nr_{L^{p,q}_T}.
\end{equation*}
\end{lem}
\begin{proof}
For a given $V$, observe that the difference $\mathbb{K}[\Phi]-\mathbb{K}[\Phi']$ involves terms of type $D_V(\phi_i\:\overline{\phi}_j,\phi_k\:\overline{\phi}_l) - D_V(\phi'_i\:\overline{\phi'}_j,\phi'_k\:\overline{\phi'}_l)$ that we manage as follows
\begin{eqnarray*}
\left|D_V(\phi_i\:\overline{\phi}_j,\phi_k\:\overline{\phi}_l)-D_V(\phi'_i\:\overline{\phi'}_j,\phi'_k\:\overline{\phi'}_l)\right| &\leq & \left|D_V([\phi_i-\phi'_i]\:\overline{\phi}_j,\phi_k\:\overline{\phi}_l)\right| +\left|D_V(\phi'_i\:[\overline{\phi}_j-\overline{\phi'}_j],\phi_k\:\overline{\phi}_l)\right|\nonumber\\
&+& \left|D_V(\phi'_i\:\overline{\phi'}_j,[\phi_k-\phi'_k]\:\overline{\phi}_l)\right| +\left|D_V(\phi'_i\:\overline{\phi'}_j,\phi'_k\:[\overline{\phi}_l-\overline{\phi'}_l])\right|.
\end{eqnarray*}
Thus, thanks to \eqref{first}, we get
\begin{eqnarray*}
&&\nr\mathbb{K}[\Phi]\:C -\mathbb{K}[\Phi']\:C'  \nr_{L^1([0,T])}  \leq  \nr\mathbb{K}[\Phi]\nr_{L^1([0,T])}\:\nr C -\:C'\nr_{\mathbb{C}^r} + \nr\mathbb{K}[\Phi]-\mathbb{K}[\Phi']\nr_{L^1([0,T])} \nr C'\nr_{\mathbb{C}^r}, \\
&& \qquad \qquad\leq \:\kappa \: T^{\frac3q-\frac12} \:\nr V\nr_{L^d} \:\nr\Phi\nr^2_{L^{\infty,2}_T}\:\nr\Phi\nr^2_{L^{p,q}_T} \: \nr C -\:C'\nr_{\mathbb{C}^r}  \\ 
&&   \qquad \qquad\:+\kappa\:\: T^{\frac3q-\frac12} \:\nr V\nr_{L^d}\:\nr C'\nr_{\mathbb{C}^r}\:\left[ \nr\Phi\nr_{L^{\infty,2}_T}\nr\Phi\nr^2_{L^{p,q}_T}+\nr\Phi'\nr_{L^{\infty,2}_T}\nr\Phi'\nr^2_{L^{p,q}_T}\right] \nr \Phi-\Phi'\nr_{L^{\infty,2}_T} \\
&&\qquad \qquad \:+  \kappa\:\: T^{\frac3q-\frac12} \:\nr V\nr_{L^d}\nr C'\nr_{\mathbb{C}^r}\:\left[ \nr\Phi\nr^2_{L^{\infty,2}_T}\nr\Phi'\nr_{L^{p,q}_T}+\nr\Phi'\nr^2_{L^{\infty,2}_T}\nr\Phi\nr_{L^{p,q}_T} \right]\nr\Phi-\Phi'\nr^2_{L^{p,q}_T}.
\end{eqnarray*}
The proof of the estimate \eqref{Lip0} follows then by setting, for instance 
 \begin{eqnarray*}
\mathrm{const}_1 &=& \kappa\:\nr V\nr_{L^d}\nr\Phi\nr^2_{L^{\infty,2}_T}\:\nr\Phi\nr^2_{L^{p,q}_T}+\kappa\:\: \:\nr V\nr_{L^d}\:\nr C'\nr_{\mathbb{C}^r}\:\left[ \nr\Phi\nr_{L^{\infty,2}_T}\nr\Phi\nr^2_{L^{p,q}_T}+\nr\Phi'\nr_{L^{\infty,2}_T}\nr\Phi'\nr^2_{L^{p,q}_T}\right] \\
&+&  \kappa\:\nr V\nr_{L^d}\nr C'\nr_{\mathbb{C}^r}\:\left[ \nr\Phi\nr^2_{L^{\infty,2}_T}\nr\Phi'\nr_{L^{p,q}_T}+\nr\Phi'\nr^2_{L^{\infty,2}_T}\nr\Phi\nr_{L^{p,q}_T} \right].
\end{eqnarray*}
Proving the inequality \eqref{Lip01} is obviously a straightforward application of \eqref{potestimate}. Now, one has   
\begin{eqnarray*}
\nr \mathbb{W}[C,\Phi]\:\Phi - \mathbb{W}[C',\Phi']\:\Phi' \nr _{L^{p',q'}_T} &\leq& \nr \mathbb{W}[C,\Phi]\:(\Phi - \Phi') \nr _{L^{p',q'}_T} \\&+&\nr \mathbb{W}[C-C',\Phi]\:\Phi' \nr _{L^{p',q'}_T} \\&+&\nr \mathbb{W}[C',\Phi-\Phi']\:\Phi' \nr _{L^{p',q'}_T}. 
\end{eqnarray*}
Observe that $ \mathbb{W}[C,\Phi] $ is quadratic in $C$ and $\Phi$. Then, by \eqref{second}, we obtain
\begin{eqnarray*}
\lefteqn{\nr \mathbb{W}[C,\Phi]\:\Phi - \mathbb{W}[C',\Phi']\:\Phi' \nr _{L^{p',q'}_T} \leq \kappa \: T^{\frac3q-\frac12}\:\nr V\nr _{L^d} \nr C\nr^2_{\mathbb{C}^r}\nr\Phi\nr^2_{L^{\infty,2}_T}\nr\Phi-\Phi'\nr_{L^{p,q}_T}   }\\&+& \kappa \: T^{\frac3q-\frac12}\:\nr V\nr _{L^d} \left[ \nr C\nr_{\mathbb{C}^r} + \nr C'\nr_{\mathbb{C}^r} \right]\nr\Phi\nr^2_{L^{\infty,2}_T}\nr\Phi'\nr_{L^{p,q}_T}\nr C-C'\nr_{\mathbb{C}^r} \\
&+& \kappa \: T^{\frac3q-\frac12}\:\nr V\nr _{L^d}  \nr C'\nr^2_{\mathbb{C}^r} \left[\nr \Phi\nr_{L^{\infty,2}_T} + \nr \Phi'\nr_{L^{\infty,2}_T}\right]\nr \Phi'\nr_{L^{p,q}_T}\nr \Phi-\Phi\nr_{L^{\infty,2}_T}.
\end{eqnarray*}
Next, setting 
\begin{eqnarray*}
\mathrm{const}_2 &=& \kappa \:\nr V\nr _{L^d} \left[ \nr C\nr^2_{\mathbb{C}^r}\nr\Phi\nr^2_{L^{\infty,2}_T}+ \nr C'\nr^2_{\mathbb{C}^r} \left[\nr \Phi\nr_{L^{\infty,2}_T} + \nr \Phi'\nr_{L^{\infty,2}_T}\right]\nr \Phi'\nr_{L^{p,q}_T} \right] \\ &+& \kappa \:\nr V\nr _{L^d} \left[ \nr C\nr_{\mathbb{C}^r} + \nr C'\nr_{\mathbb{C}^r} \right]\nr\Phi\nr^2_{L^{\infty,2}_T}\nr\Phi'\nr_{L^{p,q}_T}.
\end{eqnarray*}
for instance, finish the proof of \eqref{Lip1}.
It remains to estimate the projection part in $L^{1,2}_T$. For that purpose, we estimate first $ [\mathbf{P}_\Phi-\mathbf{P}_{\Phi'}]\: \xi $ in $L^{1,2}_T$ for a given function $\xi(t,x)\in L^{p',q'}_T$ and all $(p,q) \in \mathcal{A}$. This can be achieved thanks to H\"older inequality in space and time as follows 
\begin{eqnarray}
\nr [\mathbf{P}_\Phi-\mathbf{P}_{\Phi'}]\: \xi\nr _{L^{1,2}_T} &\leq & \sum_{k=1}^K\left[\:\nr \langle \xi,\phi_k \rangle\nr_{L^1([0,T])}\: \nr \phi_k -\phi'_k\nr_{L^{\infty,2}_T} +  \nr \langle \xi,\phi_k-\phi'_k \rangle\nr_{L^1([0,T])}\: \nr \phi'_k\nr_{L^{\infty,2}_T}\right],\nonumber\\
&\leq & \kappa \:\nr \xi\nr_{L^{p',q'}_T} \nr \Phi \nr_{L^{p,q}_T}\:\nr \Phi-\Phi'\nr_{L^{\infty,2}_T} +\kappa \:\nr \xi\nr_{L^{p',q'}_T} \nr \Phi' \nr_{L^{\infty,2}_T} \:\nr \Phi-\Phi'\nr_{L^{p,q}_T}.\label{projestimate}
\end{eqnarray}
Now, we have 
\begin{eqnarray*}
\nr \mathbf{P}_\Phi\:\mathbb{W}[C,\Phi]\:\Phi-\mathbf{P}_{\Phi'}\:\mathbb{W}[C',\Phi']\:\Phi'\nr _{L^{1,2}_T} &\leq &\nr \mathbf{P}_\Phi\:[\mathbb{W}[C,\Phi]\:\Phi-\mathbb{W}[C',\Phi']\:\Phi'\nr _{L^{1,2}_T} \\
&+& \nr [\mathbf{P}_\Phi-\mathbf{P}_{\Phi'}]\:\mathbb{W}[C',\Phi']\:\Phi'\nr _{L^{1,2}_T},\\
& \leq &\kappa \:\nr \mathbb{W}[C,\Phi]\:\Phi-\mathbb{W}[C',\Phi']\:\Phi'\nr_{L^{p',q'}_T} \nr \Phi\nr_{L^{p,q}_T}\nr \Phi \nr_{L^{\infty,2}_T} \\ 
&+& \kappa \:\nr \mathbb{W}[C',\Phi']\:\Phi'\nr_{L^{p',q'}_T}[\nr \Phi \nr_{L^{p,q}_T} + \nr \Phi' \nr_{L^{\infty,2}_T}]\nr \Phi-\Phi'\nr_{\mathcal{Z}^{p,q}_T}.
\end{eqnarray*}
The estimate above is an application of \eqref{projestimate}, first with $\Phi'\equiv 0,\:\xi \equiv \mathbb{W}[C,\Phi]\:\Phi-\mathbb{W}[C',\Phi']\:\Phi'$ and second with $\xi \equiv \mathbb{W}[C',\Phi']\:\Phi'$. Finally, Let $q =\frac{2d}{d-1}$ and recall that $V\in L^d$ with $d>\frac32$. Then  following \eqref{Lip1}, we get the desired estimate by setting for instance 
\begin{eqnarray*}
\mathrm{const}_3&=& \mathrm{const}_2\:\nr \Phi \nr_{L^{p,q}_T}\nr \Phi \nr_{L^{\infty,2}_T}+\kappa \nr V\nr_{L^d} \nr C'\nr^2_{\mathbb{C}^r} \nr \Phi' \nr_{L^{p,q}_T}\nr \Phi' \nr^2_{L^{\infty,2}_T}[\nr \Phi \nr_{L^{p,q}_T}+\nr \Phi' \nr_{L^{\infty,2}_T}],
\end{eqnarray*}
which achieves the proof.
\end{proof}
Now, let $R,T>0$ be arbitrary reals to be fixed later on. Moreover let $(C^0,\Phi^0) \in\mathcal{F}_{N,K}$ and introduce the closed ball 
\begin{equation}\label{Ball}
\tilde{\mathcal{X}}^{p,q}_T(R)= \left\{(C,\Phi) \in \mathcal{X}_T^{p,q}\::\quad \nr C,\Phi\nr_{\mathcal{X}_T^{p,q}} \leq R\right\}.
\end{equation}
This defines a complete metric space equipped with the distance induced by the norm of $\mathcal{X}_T^{p,q}$. Finally, introduce the following mapping  
\begin{eqnarray}
\pi_{C^0,\Phi^0}:\begin{bmatrix}C(\cdot) \\ \\ \Phi(\cdot) \end{bmatrix} \mapsto\begin{bmatrix}C^0 \\ \\ \mathbf{U}(\cdot)\:\Phi^0\end{bmatrix}-i\int_0^\cdot\begin{bmatrix}\mathbb{K}[\Phi(s)]\,C(s)\\ \\  \mathbf{U}(\cdot-s)\:\Big[U\:\Phi(s) +\mathbb{L}[C(s),\Phi(s)]\Big]
\end{bmatrix} \,ds\label{duhamel1},
\end{eqnarray}
with 
\[\mathbb{L}[C,\Phi] := \G[C]^{-1}\:(I-\mathbf{P}_\Phi)\:\mathbb{W}[C,\Phi]\:\Phi.\]
This formulation is then well defined as long as the matrix $\G[C(t)] $ is invertible. From now on, we will consider initial data $(C^0,\Phi^0) \in \partial\mathcal{F}_{N,K}$, that is, the associated first order density operator is of full rank. Thus, the quadratic dependence of $\G[C]$ on the coefficients $C_\sigma$ and the continuity of the MCTDHF's flow guarantees the propagation of this property up to a certain time $T^\star>0$, at least infinitesimal but possibly infinite. That is, $\G[C(t)]$ is of rank $K$ for all $t\in [0,T^\star)$, hence invertible and we refer the reader to \cite{l2Bardos12} for more details on this point. 
\vskip6pt
 Now, we claim the following
\begin{lem}\label{ContractionLemma}
Let $U,{V} \in L^d, d> \frac32, (p,q=\frac{2d}{d-1})\in \mathcal{A} $ and $(C^0,\Phi^0) \in \partial\mathcal{F}_{N,K}$ with $\nr (C^0,\Phi^0)\nr_{\mathbb{C}^r\times L^2} \leq \beta$. Then, there exist  a radius $R>0$ and a time $T$ with $0<T<T^\star$ such that the mapping $\pi $ is a strict contraction on $\tilde{\mathcal{X}}^{p,q}_T(R)$. Moreover, given $(C'^0,\Phi'^0) \in \partial\mathcal{F}_{N,K}$ with $\nr C'^0,\Phi'^0\nr_{\mathbb{C}^r\times L^2} \leq \beta$ ,then
\begin{equation}\label{continiousdependence}
\nr (C,\Phi)-(C',\Phi') \nr_{\mathcal{X}_T^{p,q}} \leq \mathrm{const}\: \nr (C^0,\Phi^0)-(C'^0,\Phi'^0) \nr_{\mathbb{C}^r\times L^2},
\end{equation}
where $(C,\Phi)$ and $(C',\Phi')$ denote the fixed points of $\pi$ associated with $(C^0,\Phi^0)$ and $(C'^0,\Phi'^0)$ respectively.
\end{lem}
\begin{proof}
The proof is based on the Lemma \ref{Lipbounds}. By abuse of notation, $\pi_{C^0,\Phi^0}(C,\Phi) $ will be used instead of the vertical notation  \eqref{duhamel1}. Moreover for notation's lightness, we set  
\[ S(\phi)(t) := \int_{0}^t \mathbf{U}(t-s)\:\phi(s)\:ds.\]
Next, let $(a,b) \in \mathcal{A} $, $(C,\Phi),({C'},{\Phi'}) \in \mathcal{X}_T^{p,q}$, $T> 0$ to be fixed later on and $t\in [0,T]$.  Finally we set $(p=\frac{4d}{3},q=\frac{2d}{d-1}) \in \mathcal{A}$. Then
\begin{eqnarray}\label{maininitialestimate}
\nr \pi_{C^0,\Phi^0}(C(t),\Phi(t)) &-& \pi_{C'^0,\Phi'^0}(C'(t),\Phi'(t))\nr_{\mathbb{C}^r\times L^{a,b}_T} \leq \nonumber\\ && (1+\rho(a))\nr(C^0,\Phi^0) - (C'^0,\Phi'^0) \nr_{\mathbb{C}^r \times L^2}\\
&+& \nr\int_0^t\:\left[\mathbb{K}[\Phi(s)]\:C(s) - \mathbb{K}[\Phi'(s)]\:C'(s)\right]\:ds \nr_{\mathbb{C}^r} + \nr S\left(U\:\Phi - U\:\Phi'\right)(t)\nr_{L^{a,b}_T}\nonumber\\
&+& \nr S\left(\mathbb{L}[C,\Phi]- \mathbb{L}[C',\Phi']\right)(t)\nr_{L^{a,b}_T}:= \mathcal{T}_1 +\ldots +\mathcal{T}_4.\nonumber
\end{eqnarray}
The term $\mathcal{T}_1$ is due to the Lemma \ref{Strilem}. More precisely, the first assertion of \eqref{S2}.  Next, thanks to the inequality \eqref{Lip0} of Lemma \ref{Lipbounds}, we have 
\begin{equation}\label{secondpart}
\mathcal{T}_2 \leq \mathrm{const}_1 \:T^{\frac3q-\frac12}\:\nr(C,\Phi)-(C',\Phi')\nr_{\mathcal{X}^{p,q}_T}.
\end{equation}
Now, we use the second assertion of \eqref{S2} in order to estimate $\mathcal{T}_3$ and $\mathcal{T}_4$. We start with 
\begin{eqnarray}\label{thirdpart}
\mathcal{T}_3 &\leq& \rho(a,p)\: \nr U\:\Phi - U\:\Phi'\nr_{L^{p',q'}_T},\nonumber\\
&\leq & \rho(a,p)\:\kappa \:T^{\frac3q-\frac12} \nr U\nr_{L^d}\:\nr(C,\Phi)-(C',\Phi')\nr_{\mathcal{X}^{p,q}_T}.
\end{eqnarray}
The second line above is due to \eqref{Lip01} where we upper-bounded obviously the $\mathcal{Z}^{p,q}_T$ norm by the $\mathcal{X}^{p,q}_T$ one. Finally
\begin{eqnarray*}
\mathcal{T}_4 &\leq&  \nr S\left(\G^[C]^{-1}\mathbb{W}[C,\Phi]\:\Phi -\G[C']^{-1}\mathbb{W}[C',\Phi']\:\Phi'\right)(t)\nr_{L^{a,b}_T} \\&+& \nr S\left(\mathbf{P}_\Phi\G[C]^{-1}\mathbb{W}[C,\Phi]\:\Phi- \mathbf{P}_{\Phi'}\G[C']^{-1}\mathbb{W}[C',\Phi']\:\Phi'\right)(t)\nr_{L^{a,b}_T}, \\
&\leq & \rho(a,p)\:\nr \G[C]^{-1}\mathbb{W}[C,\Phi]\:\Phi- \G[C']^{-1}\mathbb{W}[C',\Phi']\:\Phi'\nr_{L^{p',q'}_T} \\&+& \rho(a,p)\:\nr \mathbf{P}_\Phi\G[C]^{-1}\mathbb{W}[C,\Phi]\:\Phi- \mathbf{P}_{\Phi'}\G[C']^{-1}\mathbb{W}[C',\Phi']\:\Phi'\nr_{L^{1,2}_T}. 
\end{eqnarray*}
Next, observe the trivial algebraic relation
\begin{eqnarray*}
\nr\G[C]^{-1} - \G[C']^{-1}\nr &\leq& \nr \G[C]^{-1} (\G[C'] - \G[C] ) \G[C']^{-1}\nr,  \\&\leq& \kappa \:\nr \G[C]^{-1}\nr \nr\G[C']^{-1}\nr (\nr C\nr + \nr C'\nr) \nr C-C'\nr_{\mathbb{C}^r}.
\end{eqnarray*}
Thus, by the mean of (\ref{second},\ref{third},\ref{Lip1},\ref{Lip2}), we get
\begin{eqnarray}\label{forthpart}
\mathcal{T}_4 &\leq& \rho(a,p)\:\mathrm{const}_4 \: T^{\frac3q-\frac12}\:\nr(C,\Phi)-(C',\Phi')\nr_{\mathcal{X}^{p,q}_T}.
\end{eqnarray}
More precisely 
\begin{eqnarray*}
\mathrm{const}_4&=& \kappa \nr V\nr_{L^d}\nr\G[C]^{-1}\nr\nr\G[C']^{-1}\nr\nr C'\nr^2_{\mathbb{C}^r} [\nr C\nr_{\mathbb{C}^r}+\nr C'\nr_{\mathbb{C}^r}][1+\nr\Phi'\nr_{L^{\infty,2}_T}\nr\Phi'\nr_{L^{p,q}_T}] \times\\ &\times& \nr\Phi'\nr^2_{L^{\infty,2}_T}\nr\Phi'\nr_{L^{p,q}_T}+ \nr\G[C]^{-1}\nr [\mathrm{const}_2+\mathrm{const}_3].
\end{eqnarray*}
Summing (\ref{secondpart}-\ref{forthpart}) and adding the sum to the first line of \eqref{maininitialestimate} leads to
\begin{eqnarray}
\lefteqn{\nr \pi_{C^0,\Phi^0}(C(t),\Phi(t)) - \pi_{C'^0,\Phi'^0}(C'(t),\Phi'(t))\nr_{\mathbb{C}^r\times L^{a,b}_T} \leq (1+\rho(a))\nr(C^0,\Phi^0) - (C'^0,\Phi'^0) \nr_{\mathbb{C}^r \times L^2}}\nonumber\\ 
 &+& \big[\mathrm{const}_1 + \rho(a,p)\:\mathrm{const}_4 +\rho(a,p)\:\kappa \:\nr U\nr_{L^d}\big]\:T^{\frac3q-\frac12}\:\nr(C,\Phi)-(C',\Phi')\nr_{\mathcal{X}^{p,q}_T}. \label{contractionestimate} 
\end{eqnarray}
Now, the inequality \eqref{contractionestimate} holds for any admissible pair $(a,b)\in \mathcal{A}$.
Therefore, we write it in the special case $(a,b)=(p,q)\in \mathcal{A}$ and then in the case $(a,b)=(\infty,2)\in \mathcal{A}$. Moreover, we set $(C^0,\Phi^0) \equiv (C'^0,\Phi'^0) $ and use $\nr C,\Phi\nr_{\mathcal{X}_T^{p,q}}, \nr C',\Phi'\nr_{\mathcal{X}_T^{p,q}} \leq R $ since $(C,\Phi),(C',\Phi') \in \tilde{\mathcal{X}}_T^{p,q}(R)$. Recall that $T< T^\star$ since the initial data is in $\partial \mathcal{F}_{N,K}$, that is, the matrix $\G[C] $ and $\G[C'] $ are invertible thus there exists $ \theta >0$ such that $\nr\G[C]^{-1}\nr_{L^\infty(0,T^\star)},\nr \G[C']^{-1}\nr_{L^\infty(0,T^\star)} \leq \theta$. This procedure leads to
\begin{eqnarray}\label{contract1}
\nr \pi_{C^0,\Phi^0}(C(t),\Phi(t)) - \pi_{C^0,\Phi^0}(C'(t),\Phi'(t))\nr_{\mathcal{X}^{p,q}_T} &\leq& \mathrm{const}_5(\kappa,R,\theta,p)\:T^{\frac3q-\frac12}\times \nonumber \\ &\times&\:\nr(C,\Phi)-(C',\Phi')\nr_{\mathcal{X}^{p,q}_T}.  
\end{eqnarray}
Equivalently, if we set $(C',\Phi')\equiv (0,0)$ and use the fact that $\nr C^0,\Phi^0\nr_{\mathbb{C}^r\times L^2} \leq \beta$.  Again, after summation
\begin{eqnarray}\label{contract2}
\nr \pi_{C^0,\Phi^0}(C(t),\Phi(t)) \nr_{\mathcal{X}^{p,q}_T} \leq  \mathrm{const}_6(\kappa,R,\theta,\beta,p)\:T^{\frac3q-\frac12}.
\end{eqnarray}
More precisely
\begin{eqnarray*}
&&\mathrm{const}_5= 2 \kappa R^4\nr{V}\nr_{L^d}[5+\rho(p)(2\theta^2R^4+(2\theta+7)\theta R^2 +5\theta)]+\rho(p)\kappa\nr U\nr_{L^d},\\
&& \mathrm{const}_6= 2(1+\rho(p)) +2\kappa R^5[\nr{V}\nr_{L^d}+\rho(p)\theta\nr{V}\nr_{L^d}(1+R^2)] +2R\rho(p)\kappa \nr U\nr_{L^d}.
\end{eqnarray*}
Thus, we choose $R$ and $T$ such that
\begin{equation}\label{radiusandtime}
T < \inf\left\lbrace  T^\star, \left[\frac{R}{\mathrm{const}_6}\right]^{\frac{2q}{6-q}}\right\rbrace,\quad   R\:\mathrm{const}_5-\mathrm{const}_6 < 0.
\end{equation}
That is 
\begin{equation}\label{neededafteron}
\mathrm{const}_6 \:T^{\frac3q-\frac12} < R \quad \text{and}\quad \mathrm{const}_5\:T^{\frac3q-\frac12} < 1.
\end{equation}
Hence, by (\ref{contract1}, \ref{contract2}), $\pi_{C^0,\Phi^0} $ is a strict contracting map on $\tilde{\mathcal{X}}_T^{p,q}$. Eventually, it remains to prove the continuous dependence on the initial data \eqref{continiousdependence}. Again, the essence is the inequality \eqref{contractionestimate}. Let $(C,\Phi)$ and $(C',\Phi')$ be the fixed points of 
$\pi_{C^0,\Phi^0}$ and $\pi_{C'^0,\Phi'^0} $ respectively. Let us write again the inequality for $(a,b)=(p,q)$ and then in the case $(a,b)=(\infty,2)$ and sum up. One obtains 
\begin{eqnarray*}
\lefteqn{\nr \pi_{C^0,\Phi^0}(C(t),\Phi(t)) - \pi_{C'^0,\Phi'^0}(C'(t),\Phi'(t))\nr_{\mathcal{X}_T^{p,q}} = \nr (C,\Phi) - (C',\Phi')\nr_{\mathcal{X}_T^{p,q}}} \\ &&\leq 2 (1+\rho(p))\nr(C^0,\Phi^0) - (C'^0,\Phi'^0) \nr_{\mathbb{C}^r \times L^2}+ \mathrm{const}_5(\kappa,R,\theta,p)\:T^{\frac3q-\frac12}\nr(C,\Phi)-(C',\Phi')\nr_{\mathcal{X}^{p,q}_T}. 
\end{eqnarray*}
Finally, the estimate above with \eqref{neededafteron} proves the continuous dependence on the initial data \eqref{continiousdependence}.
\end{proof}
Now, we propose the following crucial proposition 
\begin{prop} \label{localexistence}
Let $(C^0,\Phi^0) \in \partial\mathcal{F}_{N,K}, U,V \in L^d(\mathbb{R}^3)$ with $d>\frac32$ and $(p,q=\frac{2d}{d-1}) \in \mathcal{A}$. Then there exists a $T(p,\nr C\nr_{\mathbb{C}^r},\nr \Phi\nr_{L^2},\nr \Phi^0\nr_{L^2} )>0$ and a unique solution $(C(t),\Phi(t)) \in \mathcal{X}^{p,q}_{T}$ to
\begin{eqnarray}\label{integralformulation}
\begin{bmatrix}C(t) \\ \\ \Phi(t) \end{bmatrix}=\begin{bmatrix}C^0 \\ \\ \mathbf{U}(t)\:\Phi^0\end{bmatrix}-i\int_0^t\begin{bmatrix}\mathbb{K}[\Phi(s)]\,C(s)\\ \\  \mathbf{U}(t-s)\:\Big[U\:\Phi(s) +\mathbb{L}[C(s),\Phi(s)]\Big]
\end{bmatrix} \,ds.
\end{eqnarray} 
 In particular, the function part of the solution satisfies
\begin{equation}\label{aprioriestimate}
\nr \Phi\nr_{L^{p,q}_{T}} \leq 2 \:\rho(p)\: \nr \Phi^0\nr_{L^2}.
\end{equation}
\end{prop}
\begin{proof}
Thanks to Lemma \ref{ContractionLemma}, there exists a $T'>0$ for which the integral formulation  \eqref{integralformulation} associated to the MCTDHF admits a unique solution $(C(t),\Phi(t)) \in \mathcal{X}_{T'}^{p,q}$. The main point then is to characterize $0< T\leq T'$ for which the property \eqref{aprioriestimate} holds. The function part of the solution satisfies 
\[
\Phi(t) =\mathbf{U}(t)\Phi^0 -i\:S(\mathbb{L}[C,\Phi])(t)
\]
Let $0< \tau \leq T \leq T' $ where $T$ is to be fixed later on and recall that $\nr \G[C(t)]^{-1}\nr_{L^\infty(0,T^\star)} \leq \theta $. Next, we estimate the $L^{p,q}_\tau $ of the right hand side and get using the estimate \eqref{S2} of Lemma \ref{Strilem}.\\
\begin{eqnarray}
\nr \Phi\nr_{L^{p,q}_\tau} &\leq& \rho(p) \:\nr\Phi^0\nr_{L^2}  +\nr S(U\:\Phi)(t) \nr_{L^{p,q}_\tau} \nonumber +\nr S(\G[C]^{-1}\:\mathbb{W}[C,\Phi]\:\Phi)(t) \nr_{L^{p,q}_\tau} \nonumber\\ &&\nonumber\\
&+& \nr S(\G[C(s)]^{-1}\:\mathbf{P}_{\Phi}\mathbb{W}[C,\Phi]\:\Phi)(t)\nr_{L^{p,q}_\tau}, \nonumber\\&&\nonumber\\
&\leq & \rho(p) \:\nr\Phi^0\nr_{L^2} + \nr U\:\Phi\nr_{L^{p',q'}_\tau}+\nr \G[C]^{-1}\nr\:\nr\mathbb{W}[C,\Phi]\:\Phi\nr_{L^{p',q'}_\tau} \nonumber \\&&\nonumber\\	&+& \nr \G[C]^{-1}\nr\:\nr\mathbf{P}_\Phi\:\mathbb{W}[C,\Phi]\:\Phi\nr_{L^{1,2}_\tau},\nonumber\\
&&\nonumber \\
&\leq & \rho(p) \:\nr\Phi^0\nr_{L^2} +\kappa \:\theta \:T^{\frac3q -\frac12} \:\nr U\nr_{L^d}\:\nr \Phi\nr_{L^{p,q}_\tau} +\kappa \:\theta \:T^{\frac3q -\frac12} \:\nr V\nr_{L^d} \:\nr C\nr^2_{\mathbb{C}^r} \:\nr \Phi\nr^2_{L^{\infty,2}_\tau} \nr \Phi\nr_{L^{p,q}_\tau}\nonumber\\ &&\nonumber\\
& + & \:\kappa \:\theta \:T^{\frac3q -\frac12} \:\nr V\nr_{L^d} \:\nr C\nr^2_{\mathbb{C}^r} \:\nr \Phi\nr^3_{L^{\infty,2}_\tau} \:\nr \Phi\nr^2_{L^{p,q}_\tau}. \label{estimatepq}
\end{eqnarray}
Next, we follow the argument of Tsutsumi in \cite{l2Tsutsumi}. That is, we choose $T$ so small so that there exists a positive number $\eta $ satisfying 
\begin{equation*}\label{timeestimateapriori}
\left\lbrace
\begin{array}{ll}
  & f(\eta,T):=\rho(p) \: \nr\Phi^0\nr_{L^2}-\eta + \kappa\:\theta\:\eta \: T^{\frac3q -\frac12}\left(\nr U\nr_{L^d}+\nr {V}\nr_{L^d}\nr C\nr^2_{\mathbb{C}^r} \:\nr \Phi\nr^2_{L^{\infty,2}_\tau}\left[1+\eta\:\:\nr \Phi\nr_{L^{\infty,2}_\tau}\right]\right)<0 \\ 
  &      \\
  &   0 < \eta \leq 2\:\rho(p)\:\nr \Phi^0\nr_{L^2}.    
\end{array}
\right.
\end{equation*}
For that purpose, it is sufficient to choose $T> 0$ such that 
\begin{equation}\label{timeestimate}
T< \inf\left\{ \left(2\:\kappa\:\theta \left[\nr U\nr_{L^d}+\nr {V}\nr_{L^d}\nr C\nr^2_{\mathbb{C}^r} \:\nr \Phi\nr^2_{L^{\infty,2}_\tau}\left[1+2\:\rho(p)\:\nr \Phi^0\nr_{L^2}\:\nr \Phi\nr_{L^{\infty,2}_\tau}\right]\right]\right)^{\frac{2q}{q-6}},T'\right\}.
\end{equation}
Then, let 
\begin{equation}\label{eta0}
\eta_0= \min_\eta \left\lbrace 0<\eta\leq 2\:\rho(p)\:\nr \Phi^0\nr_{L^2}\quad;\quad f(\eta,T)=0  \right\rbrace.
\end{equation}
Now, let $\mathcal{Y}(\tau) = \nr \Phi\nr_{L^{p,q}_\tau}$. Then, by \eqref{estimatepq}, we have for $0\leq \tau \leq T$ 
\begin{equation}\label{timesequence}
\left\lbrace
\begin{array}{ll}
  & \mathcal{Y}(\tau)\leq \rho(p) \: \nr\Phi^0\nr_{L^2}+ \kappa\:\theta\:\mathcal{Y}(\tau) \: T^{\frac3q -\frac12}\left(\nr U\nr_{L^d}+\nr {V}\nr_{L^d}\nr C\nr^2_{\mathbb{C}^r} \:\nr \Phi\nr^2_{L^{\infty,2}_\tau}\left[1+\mathcal{Y}(\tau)\:\:\nr \Phi\nr_{L^{\infty,2}_\tau}\right]\right) ,  \\
  &      \\
  &   \mathcal{Y}(\tau=0)=0.    
\end{array}
\right.
\end{equation}	
Then, if $T$ is chosen such that \eqref{timeestimate} holds, then by (\ref{eta0}, \ref{timesequence}), we get 
\[
 \mathcal{Y}(\tau)\leq \eta_0 \leq 2\:\rho(p)\:\nr \Phi^0\nr_{L^2}, \quad \text{for all}\quad \tau \in [0,T].
\] 
Thus, passing to the limit as $\tau\rightarrow T $, we get by Fatou's Lemma
\begin{equation*}\label{estimateresult}
\nr \Phi\nr_{L^{p,q}_T} \leq 2\:\rho(p)\:\nr \Phi^0\nr_{L^2},
\end{equation*}
which is the desired inequality.
\end{proof}
\section{proof of Theorem \ref{L2thm}}
This section is devoted to the proof of the main result Theorem \ref{L2thm}. It relies mainly on proposition \ref{localexistence} and Theorem \ref{H1thm} which assures the existence and uniqueness of a solution, in short time, to the MCTDHF system with function part in the Sobolev space $H^1$. 
 \vskip6pt
From the one side, let $(C^0,\Phi^0) \in \partial\mathcal{F}_{N,K}$ by proposition \ref{localexistence}, there exists a certain time $T_0>0$ such that the integral formulation \eqref{integralformulation} admits a unique solution $(C(t),\Phi(t))$ in the space $\mathcal{X}^{p,q}_{T_0}$.\vskip6pt
Now, one writes one more time the Duhamel formula associated to the function part of the system $\mathcal{S}$, more precisely, the second component of \eqref{duhamel}. Next, one uses the estimates (\ref{potestimate},\ref{second},\ref{third}) of Corollary \ref{corestimates} , as in \eqref{maininitialestimate} with $(C',\Phi')\equiv (0,0)$, and gets thanks to continuous embedding
\[L^{q'}(\mathbb{R}^3) \hookrightarrow H^{-1}(\mathbb{R}^3), \quad \text{for all} \quad(p,q)\in \mathcal{A} \]
and the absolute continuity of the integral (the Duhamel formulation) that 
\[\Phi \in C([0,T_0],H^{-1}) \quad \text{thus}\quad \Phi\in C_{w}([0,T_0], L^2)\]
 where the subscript $_w$ stands for weak. 
 \vskip6pt
From the opposite side, let $(C^{0,n},\Phi^{0,n})\in  \partial \mathcal{F}_{N,K} \cap H^1(\mathbb{R}^3)^K $ be a sequence of initial data for $ n\in \mathbb{N}$ such that 
\begin{equation}\nr C^{0,n},\Phi^{0,n} \nr_{\mathbb{C}^r\times L^2}  \leq \nr C^{0},\Phi^{0} \nr_{\mathbb{C}^r\times L^2},\quad  (C^{0,n},\Phi^{0,n}) \stackrel{n\rightarrow +\infty}{ \longrightarrow}\: (C^{0},\Phi^{0})\quad \text{in}\quad \mathbb{C}^r \times L^2.\label{sequence}\end{equation} \\
The Theorem \ref{H1thm} guarantees the existence of a time $\tilde T_n >0$, the existence and uniqueness of a solution in the interval of time $t\in [0,\tilde T_n) $ such that 
\begin{equation}\label{continuity}C^n\in C^1([0,\tilde T_n), \mathbb{C})^r,\quad \Phi^n\in C^0([0,\tilde T_n), H^1(\mathbb{R}^3))^K.\end{equation}\\
This solution satisfies the following conservation laws
\[
\nr C^n(t)\nr_{\mathbb{C}^r} = \nr C^{0,n}\nr_{\mathbb{C}^r},\quad \nr \Phi^n(t)\nr_{L^2}=\nr \Phi^{0,n}\nr_{L^2},\quad \forall\: t\in [0,\tilde T_n),
\] \\
for all $n\in \mathbb{N}$. Moreover, the solutions $(C^n,\Phi^n)$ being continuously depending on the initial data, one can assume (up to the extraction of a subsequence) that the sequence $\tilde T_n$ converges towards $T^\star$ where $T^\star$ denotes the maximal time of wellposedness of the Cauchy problem associated to the system $\mathcal{S}$ with initial data $(C^0,\Phi^0)$.
\vskip6pt
Then, by \eqref{sequence}, $(C^0,\Phi^0)$ and $(C^{0,n},\Phi^{0,n}) $ are fixed points of the functions $\pi_{C^{0},\Phi^{0}}$ and $\pi_{C^{0,n},\Phi^{0,n}}$ respectively on the same closed ball of $\mathcal{X}^{p,q}_{T_0}$ for all $(p,2\leq q=\frac{2d}{d-1} < 6) \in \mathcal{A}$. Thus, one can pass to the limit as $n\rightarrow +\infty $ using the continuous dependence on the initial data \eqref{continiousdependence} and \eqref{sequence}. We get
\begin{equation}\label{theestimate}\nr C(t)\nr_{\mathbb{C}^r}=\nr C^0\nr_{\mathbb{C}^r},\quad \nr\Phi(t)\nr_{L^2}= \nr\Phi^0\nr_{L^2},\quad \text{for all} \quad 0\leq t\ <\inf\:\left\{T_0,\:T^\star\right\}.
\end{equation}
Hence, \eqref{theestimate} with $\Phi\in C_{w}([0, \inf\:\left\{T_0,\:T^\star\right\}), L^2) $ leads to $\Phi\in C([0, \inf\:\left\{T_0,\:T^\star\right\}), L^2)$. Moreover, the fact that $C\in C^1([0, \inf\:\left\{T_0,\:T^\star\right\}),\mathbb{C})^r$ follows from the continuity of $t\mapsto \mathbb{K}[\Phi(t)] $ for all $t\in [0, \inf\:\left\{T_0,\:T^\star\right\})$ which is a consequence of the continuity of the $\phi_i's$.
\vskip6pt
Eventually, recall  that the function part of this solution satisfies the \textit{a priori} estimate \eqref{aprioriestimate}. In particular, the time $T_0$ depends only on the constants $N,K,p,\nr C^0\nr_{\mathbb{C}^r}$  and $\nr \Phi^0\nr_{L^2}$. Thus, one can reiterate the argument with initial data $(C(\inf\:\left\{T_0,\:T^\star\right\}),\Phi(\inf\:\left\{T_0,\:T^\star\right\})), (C(2\:\inf\:\left\{T_0,\:T^\star\right\}),\Phi(2\:\inf\:\left\{T_0,\:T^\star\right\})), \ldots$ up to $T^{\star-}$. We now check the uniqueness of the solution to the integral formulation associated to the MCTDHF, namely \eqref{duhamel}. The uniqueness on $[0,\inf\:\left\{T_0,\:T^\star\right\}]$ is given for free by the contraction argument. Again, one can reiterate up to $T^{\star-}$ and get the uniqueness on the whole interval $[0,T^\star)$. 
The points $\mathrm{i})$ and $\mathrm{ii})$ are straightforward and the last point to make clear before finishing the proof of the Theorem \ref{L2thm}, is the equivalence between the MCTDHF as an integral formulation and in the distributional sense. For that purpose, we refer, for instance, to \cite{l2Castella} (paragraph 4: Proof of Theorem 2.2, more precisely, the uniqueness part of the proof).
\vskip6pt
Recall that the Coulomb potential satisfies after cut-off $\frac{1}{|x|} \in L^d +L^\infty $ with $ d\in \left[\frac{3}{2},3\right[$. Thus, for $p=\frac{2d}{d-1}$, we have obviously (omitting the case $d=\frac{3}{2}$) $3< q < 6$ which is the result obtained by Castella in \cite{l2Castella} for the SPS system. Moreover, notice that the estimates of the Lemma \ref{Lipbounds} hold true for bounded potentials. In particular in such case one can show that the mapping 
\begin{eqnarray*}
\begin{bmatrix} C \\ \\ \Phi \end{bmatrix}\mapsto\begin{bmatrix}\mathbb{K}[\Phi]\,C \\ \\ U\:\Phi +\G[C]^{-1}\:(I-\mathbf{P}_\Phi)\:\mathbb{W}[C,\Phi]\:\Phi(s)
\end{bmatrix} \,ds,
\end{eqnarray*} 
is locally lipschitz in $\mathbb{C}^r \times L^2$. This shows in particular the restriction of the results obtained in \cite{l2Lubich1} since in this configuration we need only to prove the preservation of $\mathcal{F}_{N,K}$ by the MCTDHF flow as long as the density matrix is of full rank. 
\section{some comments and conclusion}
For the time being, all the well-posedness results on the multiconfiguration models \cite{l2Lubich1,l2trabelsi,l2Bardos1,l2Bardos12} hold up to a certain time $T^\star$, possibly infinite, for which  $\mathrm{rank}(\G[C(T^\star)]) < K$. In \cite{l2Bardos1,l2Bardos12}, we obtained with C. Bardos and I. Catto a sufficient condition that ensures the global well-posedness. More precisely, starting with initial data $(C^0,\Phi^0) \in \mathcal{F}_{N,K} \cap H^1(\mathbb{R}^3)^K$ satisfying
\begin{equation*}
\biggl\langle \Psi_{\mathrm{MC}}(C^0,\Phi^0)|{H}_N|\Psi_{\mathrm{MC}}(C^0,\Phi^0)\biggr\rangle_{L^2(\mathbb{R}^{3N})} < \min_{(C,\Phi)\in \mathcal{F}_{N,K-1}}\biggl\langle \Psi_{\mathrm{MC}}(C,\Phi)|{H}_N|\Psi_{\mathrm{MC}}(C,\Phi)\biggr\rangle_{L^2(\mathbb{R}^{3N})},
\end{equation*}
then, the MCTDHF system $\mathcal{S}$ admits a unique global-in-time solution $(C(t),\Phi(t))$ in the class 
\begin{equation*}
C^1\big([0,+\infty);\mathbb{C}\big)^r \times
\Big(C^1\big([0,+\infty);H^{-1}(\mathbb{R}^3)\big)^K \cap 
C^0\big([0,+\infty);H^{1}(\mathbb{R}^3)\big)^K\Big).
\end{equation*} 
However, such a condition is not possible for an $L^2$ theory since the energy of the initial data is not well-defined for $(C^0,\Phi^0) \in \mathcal{F}_{N,K}$. 
In order to remediate this problem in practice in numerical simulations, physicists and chemists 
use to perturb the density matrix \cite{l2Meyer,l2Scrinzi} (see also \cite{l2Lubich1}). 
In order to illustrate their method we use, without loss of generality, the following perturbation 
\[\G_\epsilon[C] = \G[C] +\epsilon\: I_K,\]
with $I_K$ being the $K\times K$ identity matrix. 
The MCTDHF system $\mathcal{S}$ (\ref{Sdef}) with $\G_\epsilon[C]$ instead of $\G[C]$ 
is then $\epsilon-$dependent. Then, observe that the $L^2$ theory we proved in the present paper is also global for $\mathcal{S}_\epsilon $ since the time $T$ of the Proposition \ref{localexistence} depends only on quantities that the flow $\mathcal{S}_\epsilon$ conserves. Consequently, one can iterate the argument in time and cover the whole real line. Notice that $\mathcal{S}_\epsilon$ does not conserve the total energy because of the perturbation. We are then able to claim the following corollary for $\mathcal{S}_\epsilon$
\begin{cor} 
Let $U,V\in L^{d}(\mathbb{R}^3)$ with $d> \frac32$ and $(C^0,\Phi^0) \in \mathcal{F}_{N,K}$ be an initial data. Then, the perturbed MCTDHF system $\mathcal{S}_\epsilon$ admits solutions $(C_\epsilon(t),\Phi_\epsilon(t)) $ satisfying 
\begin{itemize}
\item $C_\epsilon \in C^1([0,+\infty), \mathbb{C})^r\quad \text{and}\quad \Phi_\epsilon \in C^0([0,+\infty), L^2(\mathbb{R}^3))^K. $ 
\end{itemize}
Moreover, for all $2\leq q < 6$ 
\begin{itemize}
\item[$\mathrm{i})$] $\Phi_\epsilon(t) \in L^{\frac{4q}{3(q-2)}}_{loc}([0,+\infty), L^q(\mathbb{R}^3))^K $.\\
\item[$\mathrm{ii})$] The solution $(C_\epsilon(t),\Phi_\epsilon(t))$ is unique in the class
\[ L^\infty([0,+\infty), \mathbb{C})^r \times L^{\infty}([0,+\infty), L^2(\mathbb{R}^3))^K\cap L^{\frac{4q}{3(q-2)}}_{loc}([0,+\infty), L^q(\mathbb{R}^3))^K. \]
\item[$\mathrm{iii})$] $(C_\epsilon(t),\Phi_\epsilon(t))\in \partial \mathcal{F}_{N,K}$ for all $t>0$. 
\end{itemize}
\end{cor}
A global $H^1$ theory on the system $\mathcal{S}_\epsilon$ is not possible since the total energy is
not conserved, of course, since  $\G[C]\:\G_\epsilon[C]^{-1}\neq I_K$. We mention that energy conservation
can be proved using a variational principle (cf \cite{l2Bardos12}) or in a direct way by working only 
on the equations of the flow $\mathcal{S}$ (\cite{l2trabelsithesis}). 
Moreover we do not believe that the perturbed energy
$$ \Bigl\langle \left(\left[\sum_{i=1}^N\:-\frac12\:\Delta_{x_i} +U(x_i)\right]\:\G_{\epsilon}[C] + \frac 12{\mathbb{W}_{C,\Phi}}\right)\:\Phi\:,\:\Phi\Bigr\rangle$$
is decaying during the time evolution. An alternative strategy will be to pass to the limit $\epsilon \rightarrow 0$, however this depends on wether or not the system $\mathcal{S}_\epsilon$ satisfies a uniform (in $\epsilon$) estimate which is related to the non-decay of the perturbed energy. We hope to come back to this point in a forthcoming work. 
\vskip6pt \noindent\textbf{Acknowledgements} This work was supported
by the Viennese Science Foundation (WWTF) via the project "TDDFT" (MA-45),
the Austrian Science Foundation (FWF) via the Wissenschaftkolleg "Differential 
equations" (W17) and the START Project (Y-137-TEC) 
and the EU funded Marie Curie Early Stage Training Site DEASE (MEST-CT-2005-021122). \\
We thank Claude Bardos and Isabelle Catto as well as Alex Gottlieb for useful discussions and comments !!
\vskip6pt 

\end{document}